\documentclass{article}     % onecolumn (second format)

\usepackage{graphicx}
\usepackage[english]{babel}
\usepackage[UTF8]{ctex}
\usepackage[letterpaper,top=2cm,bottom=2cm,left=3cm,right=3cm,marginparwidth=1.75cm]{geometry}
\usepackage{tikz}
\usetikzlibrary{intersections}
\usetikzlibrary{arrows.meta,decorations,calligraphy}
\usepackage{graphicx}
\usepackage[colorlinks=true, allcolors=blue]{hyperref}
\usepackage{amsfonts}
\usepackage{amstext}
\usepackage{float}
\usepackage{amssymb,amsmath,amsbsy,amsthm}
\usepackage{verbatim}
\usepackage{thm-restate}
\newcommand{\bt}{\begin{tikzpicture}}
\newcommand{\et}{\end{tikzpicture}}
\providecommand{\keywords}[1]
{
  \small	
  \textbf{\textit{Keywords---}} #1
}
\newtheorem{theorem}{Theorem}
\newtheorem{lemma}{Lemma}
\newtheorem{proposition}{Proposition}

\newtheorem{cj}{Conjecture}

\tikzset{
	arr/.style={-stealth,shorten >=4.2mm,shorten <=4.2mm,thick}, %箭头样式
	dot2/.style={rotate=90,font=\LARGE}, %省略号样式1
	dot/.style={font=\LARGE} %省略号样式2
}

\begin{document}

\title{Lollipop and Cubic Weight Functions for Graph Pebbling}

\author{Marshall Yang\thanks{Brooklyn, New York}      \qquad
         Carl Yerger \thanks{Davidson College, Davidson, NC 28035
Email: {\tt cayerger@davidson.edu}} \qquad
			Runtian Zhou \thanks{
Duke University, Durham, NC 27708
Email: {\tt rz169@duke.edu}}
}

\date{\today}
% The correct dates will be entered by the editor

\maketitle

\begin{abstract}
Given a configuration of pebbles on the vertices of a graph $G$, a pebbling move removes two
pebbles from a vertex and puts one pebble on an adjacent vertex. The pebbling number of a graph $G$ is the smallest number of pebbles required such that, given an arbitrary initial configuration of pebbles, one pebble can be moved to any vertex of $G$ through some sequence of pebbling moves. Through constructing a non-tree weight function for $Q_4$, we improve the weight function technique, introduced by Hurlbert and extended by Cranston et al., that gives an upper bound for the pebbling number of graphs. Then, we propose a conjecture on weight functions for the $n$-dimensional cube. We also construct a set of valid weight functions for variations of lollipop graphs, extending previously known constructions.

\keywords{Pebbling \and Cube \and Lollipop Graph \and Weight Function}

\end{abstract}

\section{Introduction}
\label{intro}

Graph pebbling is a combinatorial game played on an undirected graph with an initial
configuration of pebbles. The game is composed of a sequence of pebbling moves, where each pebbling move removes two pebbles from one vertex and places one pebble on an adjacent vertex. The graph pebbling model was first introduced by Chung \cite{Chung}. The graph pebbling problem originated as a proof technique to prove a zero-sum theorem, introduced by Erd\H{o}s et al. \cite{ZivTheoremIT}. The details of the proof can be found in \cite{Chung}. Pebbling has been extensively studied in the last $30$ years and there have been many interesting results \footnote{\href{https://www.people.vcu.edu/~ghurlbert/pebbling/pebb.html}{https://www.people.vcu.edu/$\sim$ghurlbert/pebbling/pebb.html}}.

Calculating pebbling numbers is difficult: Clark and Milans \cite{milans_complexity_2006} proved that determining if the pebbling number is at most $k$ is $\Pi_{2}^{p}$-complete. Hurlbert \cite{hurlbert2017} introduced a linear programming technique based on weight functions on trees, in the hope of more efficiently computing bounds on pebbling numbers. Weight functions on trees might not always be able to give tight bounds for the pebbling numbers of particular graphs. For instance, Chung proved in \cite{Chung} that $\pi(Q_3)=8$, where $Q_3$ denotes the $3$-cube, but Hurlbert \cite{hurlbert2017} claimed that the best upper bound of $\pi(Q_3)$ that could be obtained using only weight functions on trees is $9$. Many researchers have attempted to generalize the known set of such functions, and each advancement brings with it a larger class of graphs on which they apply. In 2015, Cranston et al. \cite{cranston2015modified} proved $\pi(Q_3) = 8$ by presenting a set of non-tree weight functions on $Q_3$. In this paper, we construct a set of new non-tree weight functions and use them to prove $\pi(Q_4) = 16$. While it may seem a small improvement from $Q_3$ to $Q_4$, but with it we unveil a conjecture that would lead to a proof for all cubes. 

The Cartesian product of two graphs $G$ and $H$, denoted $G \mathbin{\square} H$, is a graph with vertices indexed by pairs of vertices (one from $G$ and one from $H$) with edges corresponding to an edge in $G$ or an edge in $H$.  Much of the theoretical research in pebbling is strongly motivated by a famous open question of Graham, concerning the pebbling number of the Cartesian product of graphs.

\begin{cj}[Graham {\cite{Chung}}] Given connected graphs $G$ and $H$, the Cartesian product graph $G \mathbin{\square} H$ satisfies
\label{conjBox}
 \[\pi(G \mathbin{\square} H)\leq \pi(G)\pi(H).\]
\end{cj}

Graham's conjecture has been resolved for specific families of graphs including products of paths \cite{Chung}, products of cycles \cite{herscovici2003graham,herscovici2008graham,Herscovici,sne}, products of trees \cite{sne}, and products of fan and wheel graphs \cite{FengKim2}.  It was also proved for specific products in which one of the graphs has the so-called \textbf{2-pebbling property} \cite{Chung,sne,wang2009graham}, namely that the number of pebbles required to place two pebbles on an arbitrary vertex is $2 \cdot \pi(G) - q + 1$, where $q$ denotes the number of vertices initially containing at least one pebble. A graph withouth the 2-pebbling property is called a \textbf{Lemke graph} \cite{lemke}. One of the major hurdles in tackling Graham's conjecture is the lack of tractable computational tools. Numerically verifying Graham's conjecture for specific graphs has been extremely difficult; as a result, there does not appear to be a discussion, let alone a consensus, regarding whether or not the conjecture is true.

It has been an interest in the community to find a proof for Chung's result \cite{Chung} that does not require the $2$-pebbling property, since there is some debate about its necessity in most of the results that verify Graham's Conjecture \cite{kenter}. Here we see the first evidence that a weight function for a graph $G$ can be extended to the graph $G'$, obtained by adding a pendant edge to some vertex $r$. This is equivalent, of course, to finding a weight function for $G$ that corresponds to the $2$-fold pebbling number of $r$; it is not known, in general, how to achieve this. Such an extension is the first step towards being able to apply weight functions to the Cartesian product of $G$ by $K_2$.

Additionally, we construct a set of non-tree weight functions on generalizations of lollipop graphs. The simplest of lollipops appear as subgraphs of many different graphs, and thus have been used successfully in crafting upper bounds in several instances~\cite{cranston2015modified}. Obtaining weight functions here for their generalizations should prove fruitful in the future.

Recently, Flocco et al. \cite{flocco} used mixed-integer linear programming (MILP) to offer certificates of pebbling bounds based on weight functions of an ideal set of subtrees in a graph. Flocco's work employs computational techniques to obtain weight functions using a new open-source computational toolkit.\footnote{\href{https://github.com/dominicflocco/Graph_Pebbling}{https://github.com/dominicflocco/Graph$\_$Pebbling}} The authors of \cite{flocco} and \cite{hurlbert2017} only worked on weight functions of trees. With additional computer programming, we anticipate that our non-tree weight functions will strengthen computational results that only employ weight functions of trees. Specifically, each non-tree weight function acts as an additional set of stronger constraints that can now be applied to any graph that has this structure as a subgraph when computing bounds for its pebbling number. Therefore, using our non-tree weight functions, the techniques such as MILP may be used to provide certificates of better upper bounds for pebbling numbers. As seen in \cite{cranston2015modified}, incorporating more elaborate weight functions can significantly improve known pebbling bounds, and we see the potential for these improvements as a major consequence of our work.

\subsection{Notation and Terminology}

We denote the vertex and the edge sets of a graph $G$ by $V(G)$ and $E(G)$ respectively. The \textbf{distance} between two distinct vertices $u,v\in V(G)$, denoted by $d(u,v)$, is the length of the shortest path between $u$ and $v$. The \textbf{diameter} of $G$ is defined as $\max_{u,v\in V(G)}d(u,v)$. A (pebble) \textbf{configuration} $p$ on $G$ is a function $p: V(G) \rightarrow \mathbb N \cup \{0\}$, where $p(v)$ is the number of pebbles on $v$ for each $v\in V(G)$. A \textbf{weight function} on a graph $G$ is a function $w: V(G)\rightarrow \mathbb{R}^+\cup\{0\}$. A \textbf{pebbling move} from vertex $u$ to an adjacent vertex $v$ removes $2$ pebbles on $u$ and places $1$ pebble on $v$. Let $1_G$ be the configuration on $G$ where each vertex has exactly $1$ pebble.

Let $G'$ be a subgraph of $G$ and $p$ be a configuration on $G$. Let $w$ be a weight function on $G$. Denote by $w_{G'}$ the weight function on $G'$ formed by the restriction of $w$ on $V(G')$. The size of $p$, denoted by $|p|$, the number of pebbles on $G'$, denoted by $p(G')$, and the weight of $p$, denoted by $w(p)$, are defined as follows:
\begin{alignat*}{3}
|p|&=\sum_{v\in V(G)}p(v)&\qquad p(G')&=\sum_{v\in V(G')}p(v)&\qquad w(p)&=\sum_{v\in V(G)}p(v)\cdot w(v).
\end{alignat*}
If $G$ is ``rooted'' at some vertex $r\in V(G)$ and there exists a possibly empty sequence of pebbling moves started from $p$ that could put $1$ pebble on $r$, we say $p$ is \textbf{$r$-solvable}. If no such sequence exists, we say $p$ is $r$-unsolvable. A weight function $w$ for a graph $G$ rooted at $r$ is \textbf{valid} if and only if $r$ is the only vertex with weight $0$ and every $r$-unsolvable configuration $p$ satisfies $w(p)\leq w(1_G)$. Note that our definition of valid weight functions is slightly different from the definition of strategy in~\cite{hurlbert2017}, as we restrict that every vertex other than $r$ has to have \textit{positive} weight. In this paper, we only discuss weight functions with $w(r)=0$, so we will not draw the weight of $r$ in figures. The \textbf{pebbling number of a graph $G$ rooted at $r$}, denoted by $\pi(G,r)$, is the smallest integer such that any configuration with at least $\pi(G,r)$ pebbles could reach $r$ through some sequence of pebbling moves. The \textbf{pebbling number} of $G$, denoted by $\pi(G)$, is defined as $$\pi(G)=\max_{r\in V(G)}\pi(G,r).$$

The \textbf{hypercube} of dimension $n$, denoted by $Q_n$, is defined by $$V(Q_n)=\{(x_1,x_2,\ldots,x_n)|x_1,x_2,\ldots,x_n\in \{0,1\}\},$$ where two vertices are adjacent if and only if they differ in exactly one coordinate.

\subsection{Background}
In this subsection we introduce some helpful lemmas that will be used in later proofs.
\begin{lemma}\label{l1}
Let $w$ be a valid weight function on a graph $G$ rooted at $r$. Let \\$m=\min_{v\in V(G)-\{r\}}w(v)$. Then, $$\pi(G,r)\leq
\left \lfloor \frac{w(1_G)}{m} \right \rfloor  +1.$$
\end{lemma}
\begin{proof}
Given any configuration $p$ with $|p|= \left \lfloor \frac{w(1_G)}{m} \right \rfloor  +1$ and $p(r)=0$, we have $$w(p)\geq m\cdot \left (\left \lfloor \frac{w(1_G)}{m} \right \rfloor  +1 \right )>w(1_G).$$ Thus by validity of $w$, configuration $p$ is $r$-solvable. Because $p$ is arbitrary, $$\pi(G,r)\leq |p|=\left \lfloor \frac{w(1_G)}{m} \right \rfloor  +1.$$
\qed
\end{proof}

Note that Lemma~\ref{l1} generalizes Corollary 3 in \cite{hurlbert2017} and Lemma 2 in \cite{cranston2015modified}.

\begin{lemma}\label{l2}
For any $k\geq 1$, let $P_{k+1}$ be the path $(v_0,v_1,v_2,\ldots, v_{k-1},r)$ rooted at $r$. Let $w$ be the weight function on $P_{k+1}$ such that $w(r)=0$ and $w(v_i)=2^i$ for each $i$ from $0$ to $k-1$. Let $p$ be an arbitrary configuration on $P_{k+1}$. Then, $p$ is $r$-solvable if and only if $w(p)\geq 2^k$. More generally, the pebbling number of graphs with diameter $k$ must be at least $2^k$.
\end{lemma}

Lemma~\ref{l2} was mentioned in Section 2 of \cite{hurlbert2017} but was first observed in the proof of Theorem 4.2 in \cite{hurlbert2002} by Czygrinow et al.. It gives an idea on how weight function techniques can be used to calculate pebbling numbers. The next lemma is a generalized form of Lemma \ref{l2}.

\begin{lemma}[Weight Function Lemma, \cite{hurlbert2017}]\label{l3}
Let $w$ be a weight function on a tree $T$ rooted at $r$. For each $v\in V(T)-\{r\}$, let $v^+$ be the parent of $v$, the neighbor of $v$ that is closer to $r$. If $w(r)=0$ and $w(v^+)\geq 2w(v)$ for every $v$ not adjacent to $r$, then $w$ is valid.
\end{lemma}
Lemma \ref{l3} presents the first set of valid weight functions described in the pebbling literature. Hurlbert \cite{hurlbert2017} devised these weight functions and Cranston et al. \cite{cranston2015modified} made improvements upon them.
\begin{lemma}[Conic Lemma]\label{l4}
For a graph $G$ rooted at $r$, let weight function $w$ on $G$ be a conic combination of valid weight functions on connected subgraphs of $G$ rooted at $r$, such that $w(v)>0$ for each $v\neq r$. Then $w$ is valid.
\end{lemma}
%In our definition, every valid weight function has w(r)=0. Conic combination of 0 is still 0.
\begin{proof}
Let $w_1,w_2,\ldots,w_m$ be valid weight functions on $G_1,G_2,\ldots,G_m$ respectively, where \\$G_1,G_2,\ldots,G_m$ are all connected subgraphs of $G$ rooted at $r$. For each $i$ from $1$ to $m$ and for each $v\in V(G)-V(G_i)$, define $w_i(v)=0$. Let $a_1,a_2,\ldots,a_m$ be non-negative real numbers such that for each $v\in V(G)$, $w(v)=\sum_{i=1}^{m}a_iw_i(v)$. Let $p$ be an arbitrary $r$-unsolvable configuration. Then, the validity of $w_1,w_2,\ldots,w_m$ implies that:
\begin{align*}
w(p)=\sum_{i=1}^{m}a_iw_i(p)\leq \sum_{i=1}^{m}a_iw_i(1_G)=w(1_G).
\end{align*}
Since $p$ is arbitrary, $w$ is a valid weight function on $G$.
\qed
\end{proof}

Note that the Conic Lemma generalizes Lemma 5 of \cite{hurlbert2017}.

When studying large graphs, it is not easy to construct valid weight functions directly. A common method is studying subgraphs of the original graph, constructing valid weight functions on those smaller subgraphs, and finally applying Lemma \ref{l4} to get a valid weight function on the original graph.\\

Here is an example of how these lemmas may be applied.
\begin{theorem}\label{t1}
For any $k\in \mathbb N$, $\pi(C_{2k+1})=2\left \lfloor \frac{2^{k+1}}{3}\right \rfloor +1$.
\end{theorem}
This theorem was first proved by Pachter et al. as Theorem 8 in \cite{opg}. It was also proved as Theorem 5 in \cite{Bunde2005} and as Theorem 11 in \cite{hurlbert2017}. Here, we reformalize the proof of the upper bound in \cite{hurlbert2017} using Lemmas~\ref{l1},\ref{l2}, and \ref{l4} in order to illustrate the technique of weight functions.

\begin{figure}[H]
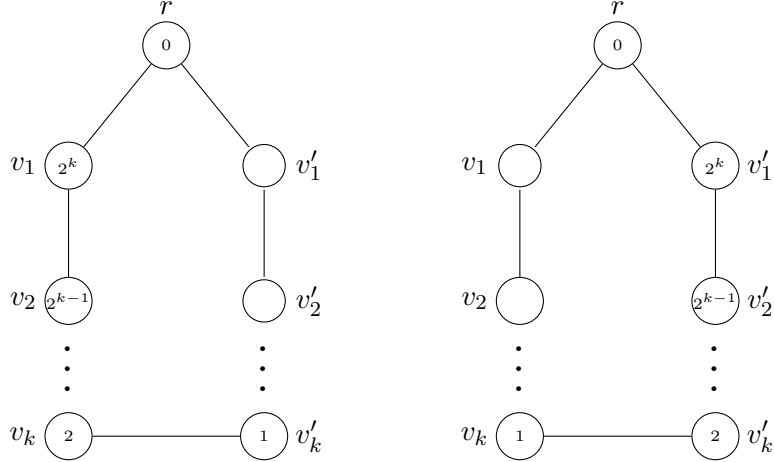

    \centering
\bt
\begin{scope}[xshift=-5cm]

\draw(-1,0)--(-1,1.5)--(0,2.3);
\draw(-1,-1.5)--(1,-1.5);
\draw[fill](-1,0)circle(2pt);
\draw(-1.5,0)node{$2^{k-1}$};
\draw[fill](-1,1.5)circle(2pt);
\draw(-1.5,1.5)node{$2^k$};
\draw[fill](-1,-1.5)circle(2pt);
\draw(-1.5,-1.5)node{$2$};

\draw[fill](1,-1.5)circle(2pt);
\draw(1.5,-1.5)node{$1$};
\draw(2,0.4)node{\large $+$};
\draw[fill](0,2.3)circle(2pt);
\draw(0,2.7)node{$r$};
\node[dot2] at (-1,-0.7) {$\cdots$};

\end{scope}
\begin{scope}[xshift=-1cm]
\draw(1,0)--(1,1.5)--(0,2.3);
\draw(-1,-1.5)--(1,-1.5);
\draw[fill](1,0)circle(2pt);
\draw(1.5,0)node{$2^{k-1}$};
\draw[fill](1,1.5)circle(2pt);
\draw(1.5,1.5)node{$2^k$};
\draw[fill](-1,-1.5)circle(2pt);
\draw(-1.5,-1.5)node{$1$};
\draw(2.97,0.4)node{\large $=$ $w:$};

\draw[fill](1,-1.5)circle(2pt);
\draw(1.5,-1.5)node{$2$};

\draw[fill](0,2.3)circle(2pt);
\draw(0,2.7)node{$r$};
\node[dot2] at (1,-0.7) {$\cdots$};
\end{scope}
\begin{scope}[xshift=5cm]

\draw(1,0)--(1,1.5)--(0,2.3)--(-1,1.5)--(-1,0);
\draw(-1,-1.5)--(1,-1.5);
\draw[fill](1,0)circle(2pt);
\draw(1.5,0)node{$2^{k-1}$};
\draw[fill](1,1.5)circle(2pt);
\draw(1.5,1.5)node{$2^k$};
\draw[fill](-1,-1.5)circle(2pt);
\draw(-1.5,-1.5)node{$3$};
\draw[fill](-1,0)circle(2pt);
\draw(-1.5,0)node{$2^{k-1}$};
\draw[fill](-1,1.5)circle(2pt);
\draw(-1.5,1.5)node{$2^k$};

\draw[fill](1,-1.5)circle(2pt);
\draw(1.5,-1.5)node{$3$};

\draw[fill](0,2.3)circle(2pt);
\draw(0,2.7)node{$r$};
\node[dot2] at (1,-0.7) {$\cdots$};
\node[dot2] at (-1,-0.7) {$\cdots$};
\end{scope}
\et
\caption{\label{f1}Valid weight functions on subgraphs of the odd cycle $C_{2k+1}$}
\end{figure}

\begin{proof}
The two weight functions in Figure \ref{f1} on $P_{k+2}$ are valid by Lemma \ref{l2}. Let weight function $w$ be the addition of the two weight functions as shown in Figure \ref{f1}. By Lemma \ref{l4}, $w$ is valid. Thus by Lemma \ref{l1}, $$\pi(C_{2k+1})=\pi(C_{2k+1},r)\leq \left \lfloor \frac{2\cdot 2^{k+1}-2}{3}\right \rfloor+1=2\left \lfloor \frac{2^{k+1}}{3}\right \rfloor +1.$$

%With not much more work, it is possible to show that $\pi(C_{2k+1})>2\left \lfloor \frac{2^{k+1}}{3}\right \rfloor$ which would complete the proof.

To prove $\pi(C_{2k+1})>2\left \lfloor \frac{2^{k+1}}{3}\right \rfloor$, consider the configuration where each of the $2$ vertices with distance $k$ from $r$ has $\left \lfloor \frac{2^{k+1}}{3}\right \rfloor$ pebbles and all other vertices have no pebbles. It can be shown that this configuration is $r$-unsolvable. For more details, please refer to Theorem 8 in \cite{opg}. 
\qed
\end{proof}

\section{Weight Functions on Cubes}

In this section we focus on cubes and construct several valid weight functions on related graphs. An \textbf{$n$-dimensional cube}, denoted by $Q_n$, consists of $2^n$ vertices labeled by $\{0,1\}$-tuples of length $n$. Cubes are common and easily found structures in graphs. Further, they are symmetric and easy to describe. They were studied by Chung in the very first paper of the pebbling literature \cite{Chung}. Hurlbert also talked about cubes when introducing the weight function technique \cite{hurlbert2017}, and in \cite{cranston2015modified}, Cranston et al. obtained the pebbling number of $Q_3$ using non-tree weight functions. In this section, we obtain the pebbling number of $Q_4$ using non-tree weight functions and give a conjecture on a set of non-tree weight functions that, if true, would determine the pebbling number for a cube of any dimension.

\begin{proposition}
The weight function in Figure \ref{f2} is valid.
\end{proposition}

\begin{figure}
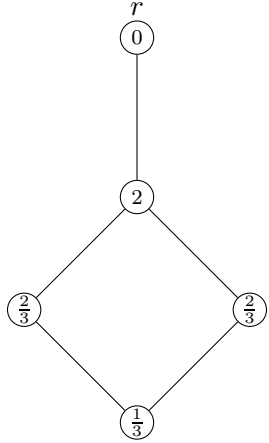

\centering
\bt
\draw(0,1.5)--(1.5,0)--(0,-1.5);
\draw(0,-1.5)--(-1.5,0)--(0,1.5)--(0,3.621);
\draw[fill](1.5,0)circle(2pt);
\draw(1.8,0)node{$2\over 3$};
\draw[fill](-1.5,0)circle(2pt);
\draw(-1.8,0)node{$2\over 3$};
\draw[fill](0,-1.5)circle(2pt);
\draw(0,-1.9)node{$1\over 3$};
\draw[fill](0,1.5)circle(2pt);
\draw(0.15,1.75)node{$2$};
\draw[fill](0,3.621)circle(2pt);
\draw(0,4)node{$r$};
\et

\caption{\label{f2}A valid weight function for a root attached to $Q_2$}
\end{figure}
\begin{proof}
Let the graph in Figure \ref{f2} be $G$ and the weight function in Figure \ref{f2} be $w$. If we delete one of the two bottom edges of $G$, we will get a new graph, denoted by $G'$. Because $G'$ is a tree and $w_{G'}(v^+)\geq 2w_{G'}(v)$ for each $v\in V(G')-\{r\}$, by Lemma \ref{l3} $w_{G'}$ is valid. Let $p$ be an arbitrary $r$-unsolvable configuration on $G$. Clearly $p$ is $r$-unsolvable on $G'$. By validity of $w_{G'}$, $$w_G(p)=w_{G'}(p)\leq w_{G'}(1_{G'})=w_{G}(1_G).$$ Since $p$ is arbitrary,  $w$ is valid on $G$.
\qed
\end{proof}

\begin{proposition}\label{p1}
Let $V(Q_3)=\{(x_1,x_2,x_3)|x_1,x_2,x_3\in \{0,1\}\}$. Let $w'$ be the weight function on $Q_3$ defined by $w'(1,0,0)=w'(0,1,0)=w'(0,0,1)=2,$ $w'(1,1,0)=w'(1,0,1)=w'(0,1,1)={4\over 3}$, $w'(1,1,1)=1$, with $(0,0,0)$ being the root, as shown in Figure~\ref{f3}. Then, $w'$ is valid.
\end{proposition}

\begin{figure}
    \centering
\begin{tikzpicture}
\begin{scope}[xshift=-4cm]
\draw(0,0)--(-2.25,1.7)--(0,3.85)--(2.25,1.7)--(0,0);
\draw(0,1.7)--(-2.25,3.85)--(0,5.55)--(2.25,3.85)--(0,1.7)--(0,0);
\draw(-2.25,1.7)--(-2.25,3.85);
\draw(2.25,1.7)--(2.25,3.85);
\draw(0,5.55)--(0,3.85);
\draw[fill](0,0)circle(2pt);
\draw (0,-0.3) node {$1$};
\draw[fill](0,3.85)circle(2pt);
\draw (0,3.5) node {$2$};
\draw[fill](0,1.7)circle(2pt);

\draw(0,-0.9)node{\large $w'$};

\draw (0,2.05) node {$4\over 3$};
\draw[fill](2.25,1.7)circle(2pt);

\draw (2.5,1.7) node {$4\over 3$};
\draw[fill](-2.25,1.7)circle(2pt);
\draw (-2.5,1.7) node {$4\over 3$};
\draw[fill](-2.25,3.85)circle(2pt);
\draw (-2.5,3.85) node {$2$};
\draw[fill](2.25,3.85)circle(2pt);
\draw (2.5,3.85) node {$2$};

\draw[fill] (0,5.55)circle(2pt);

\draw (0,5.8) node {$r$};
\end{scope}

\begin{scope}[xshift=4cm]
\draw(0,0)--(-2.25,1.7)--(0,3.85)--(2.25,1.7)--(0,0);
\draw(0,1.7)--(-2.25,3.85)--(0,5.55)--(2.25,3.85)--(0,1.7)--(0,0);
\draw(-2.25,1.7)--(-2.25,3.85);
\draw(2.25,1.7)--(2.25,3.85);
\draw(0,5.55)--(0,3.85);
\draw[fill](0,0)circle(2pt);
\draw (0,-0.3) node {$(1,1,1)$};
\draw[fill](0,3.85)circle(2pt);
\draw (0,3.25) node {$(0,1,0)$};
\draw[fill](0,1.7)circle(2pt);

\draw (0,2.3) node {$(1,0,1)$};
\draw[fill](2.25,1.7)circle(2pt);

\draw(0,-0.9)node{\large $Q_3$};

\draw (2.9,1.7) node {$(0,1,1)$};
\draw[fill](-2.25,1.7)circle(2pt);
\draw (-2.9,1.7) node {$(1,1,0)$};
\draw[fill](-2.25,3.85)circle(2pt);
\draw (-2.9,3.85) node {$(1,0,0)$};
\draw[fill](2.25,3.85)circle(2pt);
\draw (2.9,3.85) node {$(0,0,1)$};

\draw[fill] (0,5.55)circle(2pt);

\draw (0,5.85) node {$(0,0,0)$};
\end{scope}
\end{tikzpicture}
    \caption{\label{f3}A valid weight function on $Q_3$, denoted by $w'$}

\end{figure}
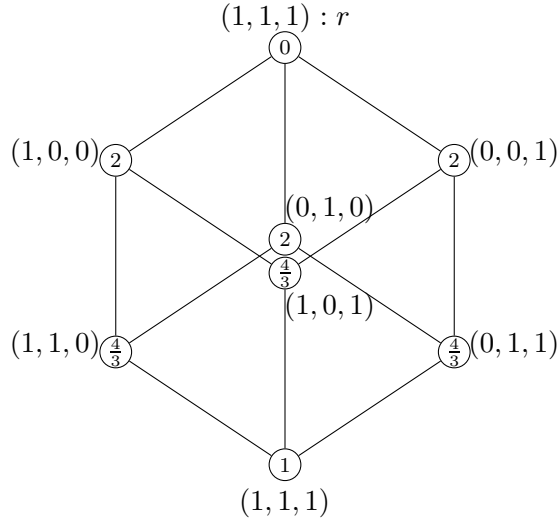

\begin{proof}
Notice that $w'$ is the addition of $3$ copies of the weight function in Proposition~\ref{p1}. In particular, consider vertex sets 
\begin{align*}
V_1=\{(1,x_2,x_3)|x_2,x_3\in \{0,1\}\}\cup\{(0,0,0)\},\\
V_2=\{(x_1,1,x_3)|x_1,x_3\in \{0,1\}\}\cup\{(0,0,0)\},\\
V_3=\{(x_1,x_2,1)|x_1,x_2\in \{0,1\}\}\cup\{(0,0,0)\}.
\end{align*}
Then, the $3$ induced subgraphs of $V_1,V_2,$ and $V_3$, with $(0,0,0)$ being root on each of them, are the graphs that these $3$ copies of the weight function in Proposition~\ref{p1} lie on. Thus by Lemma \ref{l4}, $w'$ is valid.
\qed
\end{proof}

\begin{lemma}\label{l5}
Let $w$ be the weight function on the graph $G$ in Figure~\ref{f4} defined by $w(x_1)=w(x_2)=w(x_3)=2,$ $w(y_1)=w(y_2)=w(y_3)={4\over 3}$, $w(u)=4$, and $w(z)=1$, with $r$ being the root. Then, $w$ is valid.
\end{lemma}
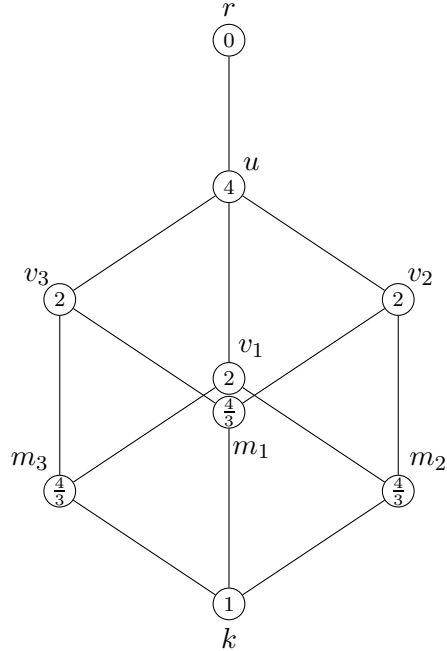
\begin{figure}
\centering
\begin{tikzpicture}
\begin{scope}[xshift=-4cm]
\draw(0,0)--(-2.25,1.7)--(0,3.85)--(2.25,1.7)--(0,0);
\draw[fill](0,5.55)--(0,7)node[anchor=-90]{$r$}circle(2pt);
\draw(0,1.7)--(-2.25,3.85)--(0,5.55)--(2.25,3.85)--(0,1.7)--(0,0);
\draw(-2.25,1.7)--(-2.25,3.85);
\draw(2.25,1.7)--(2.25,3.85);
\draw(0,5.55)--(0,3.85);
\draw[fill](0,0)circle(2pt);
\draw (0,-0.3) node {$1$};
\draw[fill](0,3.85)circle(2pt);
\draw (0,3.5) node {$2$};
\draw[fill](0,1.7)circle(2pt);

\draw(0,5.55)node[anchor=225]{$4$};

\draw(0,-0.9)node{\large $w$};

\draw (0,2.05) node {$4\over 3$};
\draw[fill](2.25,1.7)circle(2pt);

\draw (2.5,1.7) node {$4\over 3$};
\draw[fill](-2.25,1.7)circle(2pt);
\draw (-2.5,1.7) node {$4\over 3$};
\draw[fill](-2.25,3.85)circle(2pt);
\draw (-2.5,3.85) node {$2$};
\draw[fill](2.25,3.85)circle(2pt);
\draw (2.5,3.85) node {$2$};

\draw[fill] (0,5.55)circle(2pt);

\end{scope}

\begin{scope}[xshift=4cm]
\draw(0,0)--(-2.25,1.7)--(0,3.85)--(2.25,1.7)--(0,0);
\draw[fill](0,5.55)--(0,7)node[anchor=-90]{$r$}circle(2pt);
\draw(0,1.7)--(-2.25,3.85)--(0,5.55)--(2.25,3.85)--(0,1.7)--(0,0);
\draw(-2.25,1.7)--(-2.25,3.85);
\draw(2.25,1.7)--(2.25,3.85);
\draw(0,5.55)--(0,3.85);
\draw[fill](0,0)circle(2pt);
\draw (0,-0.3) node {$z$};
\draw[fill](0,3.85)circle(2pt);
\draw (0,3.5) node {$x_2$};
\draw[fill](0,1.7)circle(2pt);

\draw (0,2.05) node {$y_2$};
\draw[fill](2.25,1.7)circle(2pt);

\draw(0,-0.9)node{\large $G$};

\draw (2.5,1.7) node {$y_1$};
\draw[fill](-2.25,1.7)circle(2pt);
\draw (-2.5,1.7) node {$y_3$};
\draw[fill](-2.25,3.85)circle(2pt);
\draw (-2.5,3.85) node {$x_1$};
\draw[fill](2.25,3.85)circle(2pt);
\draw (2.5,3.85) node {$x_3$};

\draw[fill] (0,5.55)circle(2pt);

\draw(0,5.55)node[anchor=225]{$u$};
\end{scope}
\end{tikzpicture}
    \caption{\label{f4}A valid weight function for a root attached to $Q_3$, denoted by $w$}
\end{figure}

\begin{proof}
We begin with terminology helpful for our proof. Let the graph in Figure \ref{f4} be $G$. Let the subgraph induced by $\{x_1,x_2,x_3,y_1,y_2,y_3,z\}$ be $G'$. Let the weight functions in Figure \ref{f4} and Figure \ref{f3} be $w$ and $w'$, respectively. Let $X=\{x_1,x_2,x_3\}$ and $Y=\{y_1,y_2,y_3\}$. For each $y_i\in Y$, define its parents, denoted by $pa(y_i)$, as the set of vertices adjacent to $y_i$ that are closer to $r$. Hence,  $pa(y_1)=\{x_2,x_3\},pa(y_2)=\{x_1,x_3\},pa(y_3)=\{x_1,x_2\}$. Let $p$ be an arbitrary configuration such that $w(p)>w(1_G)=15$. By definition of validity in Section 1.1, it suffices to prove that $p$ is $r$-solvable. Without loss of generality, assume $p(y_1)\geq p(y_2)\geq p(y_3)$. For purposes of proof by contradiction, assume $p$ is $r$-unsolvable. Clearly $p(u)\leq 1$.\\

\textbf{Case 1:} Suppose that $p(u)=1$.

If $p(u)=1$, then $w_{G'}(p)=w(p)-w(u)\cdot p(u)>15-4=11$. By Proposition~\ref{p1} we can use pebbles on $G'$ to move another pebble to $u$, then from $u$ a pebble can be moved to $r$. This shows $p$ is $r$-solvable. \\

In the rest of the cases, suppose that $p(u)=0$.\\

\textbf{Case 2:} Suppose that $p(x_i)\geq 2$ for some $x_i\in X$.

Move $1$ pebble from $x_i$ to $u$ to get a new configuration $p'$. Clearly $w(p')=w(p)>15$. By Case $1$, $p'$ is $r$-solvable. Thus $p$ is also $r$-solvable. \\

In the rest of the cases, suppose that $p(u)=0$ and $p(x_i)\leq 1$ for each $x_i\in X$.\\

\textbf{Case 3:} Suppose that $p(x_i)=1$ for each $x_i\in V$.

Observe that if we can move $2$ pebbles from $\{y_1,y_2,y_3,z\}$ to $X$, then we can move them to two different vertices and $p$ is immediately $r$-solvable. Suppose we can move at most $1$ pebble from $Y$ to $X$. Then, $p(Y)\leq 1\cdot2+3\cdot 1=5$. Thus $$p(z)\geq {w(p)-5\cdot \frac{4}{3}-3\cdot 2\over 1}>2$$ so we can move $1$ pebble from $z$ to any vertex of $Y$. But $p(z)>2$ implies that $p(Y)\leq 3$, otherwise we could move $2$ pebbles from $\{y_1,y_2,y_3,z\}$ to $X$. Following similar logic, $p(Y)\leq 3$ implies that $p(z)\geq 6$, $p(z)\geq 6$ implies that $p(Y)= 0$, and $p(Y)= 0$ implies that $p(z)\geq 8$. Finally, $p(z)\geq 8$ implies that we can move two pebbles from $z$ to $X$, which shows $p$ is $r$-solvable.\\

 In the rest of the cases, suppose that $p(x_i)=0$ for some $x_i\in X$.\\

\textbf{Case 4:} Suppose that $w(p)\geq 16$.

\textbf{Subcase 4.1:} Suppose that some vertex in $X$ has $0$ pebbles and the two other vertices of $X$ have $1$ pebble. We will show the proof when $p(x_1)=p(x_2)=1$ and $p(x_3)=0$.

If $p(y_i)>2$ for some $y_i\in Y$, we can move one pebble from $y_i$ to some $x_j\in pa(y_i)$ such that $p(x_j)=1$ to form a new configuration $p'$. Notice that $w(p')=w(p)-{2\over 3}>15$. Thus by Case $2$, $p'$ is $r$-solvable.

If $p(y_i)\leq 1$ for each $y_i\in Y$, then $p(z)\geq 16-2\cdot 2-3\cdot \frac{4}{3}=8$. From $z$ we can move $1$ pebble to each of $x_1$ and $x_2$ and clearly $p$ is $r$-solvable.\\

\textbf{Subcase 4.2:} Suppose that $p(y_j)\geq 2$ and $p(x_i)=1$ for some $y_j\in Y$ and $x_i\in pa(y_j)$.

Clearly, we can remove these $3$ pebbles to place a pebble on $u$ to form a new configuration $p'$. Notice that $w_{G'}(p')=w(p)-1\cdot 2-2\cdot\frac{4}{3}>11$. By Proposition \ref{p1}, we can move another pebble from $G'$ to $u$, so $p$ is $r$-solvable.\\

In the next two subcases, suppose that for each $y_i,y_j\in Y$, if their common parent has exactly one pebble, then $p(y_i)\leq 1$ and $p(y_j)\leq 1$.\\

\textbf{Subcase 4.3:} Suppose that $p(x_i)=0$ for each $x_i\in X$.

Because $p$ is $r$-unsolvable and $p(y_1)\geq p(y_2)\geq p(y_3)$, the following two implications are clear:
\begin{align*}
p(y_1)\geq 6&\Rightarrow p(y_2)\leq 1\textnormal{ and } p(y_3)\leq 1\\
p(y_1)\geq 4&\Rightarrow p(y_2)+p(y_3)\leq 4.
\end{align*}
These two implications further imply that $Y$ has at most $9$ pebbles. Since $w(z)=1$, we know $p(z)\geq \left \lceil 16-9\cdot\frac{4}{3}\right \rceil =4$ and we can move $2$ pebbles from $z$ to any vertex of $Y$.

If $p(y_1)+p(y_2)=6$, moving $2$ pebbles from $z$ to $Y$ can make $p(y_1)$ and $p(y_2)$ both even and they can move $\frac{6+2}{2}=4$ pebbles to $pa(y_1)\cap pa(y_2)$, and $p$ is clearly $r$-solvable. Hence $p(y_1)+p(y_2)<6$. Since $p(y_1)\geq p(y_2)\geq (y_3)$, we also know $p(y_3)\leq 2$.

Notice that if $p(y_1)+p(y_2)=5$ and $p(y_3)=2$, then $p(y_1)=3$ and $p(y_2)=p(y_3)=2$. But since $p(z)=4$, $p$ is $r$-solvable. Hence $p(Y)\leq 6$ and $p(z)\geq \left \lceil 16-6\cdot \frac{4}{3}\right \rceil=8$. If $p(Y)\geq 4$, moving $4$ pebbles from $z$ to $Y$ could make each $p(y_i)$ even, and then $p(Y)=8$ forces $p$ to be $r$-solvable. Thus $p(Y)\leq 3$. Following similar logic, $p(Y)\leq 3$ implies that $p(z)\geq 12$, $p(z)\geq 12$ implies that $p(Y)\leq 1$,  $p(Y)\leq 1$ implies that $p(z)\geq 14$, and $p(z)\geq 14$ implies that $p(Y)=0$. Finally, $p(Y)=0$ implies that $p(z)\geq 16$, forcing $p$ to be $r$-solvable.\\

\textbf{Subcase 4.4:} Suppose that some vertex $x_i\in X$ has $1$ pebble and the two other vertices of $X$ have $0$ pebbles.

Let $y_j,y_k,y_t\in Y$ such that $x_i\in pa(y_j)\cap pa(y_k)$ and $x_i\not \in pa(y_t)$. We know by assumption that $p(y_j)\leq 1,p(y_k)\leq 1$. Thus $p(y_t)\cdot w(y_t)+p(z)\cdot w(z)\geq 16-1\cdot 2-2\cdot\frac{4}{3}=\frac{34}{3}$.

If $p(y_j)=p(y_k)=0$, then $p(y_t)\cdot w(y_t)+p(z)\cdot w(z)\geq 14$. Observe that $p(x_i)=1$ implies that $p(z)\leq 11$, so $p(y_t)\geq \left \lceil \frac{14-11}{\frac{4}{3}}\right \rceil=3$. This implies that $p(y_t)\cdot w(y_t)+p(z)\cdot w(z)\geq 14-3\cdot\frac{4}{3}=10$. Move pebbles from $z$ to $y_t$ until at most $1$ pebble is left at $z$. Since $2>\frac{4}{3}$, $y_t$ now has at least $3+\left \lceil \frac{10-1}{2}\right \rceil=8$ pebbles and $p$ is $r$-solvable.

If $p(y_j)=1$ and $p(y_k)=0$, then $p(z)\leq 9$, so $w(y_t)\geq 16-2-9-\frac{4}{3}=\frac{11}{3}$ and $p(y_t)\geq 3$. Following similar logic, $p(y_t)\geq 3$ implies that $p(z)\leq 3$. Finally, $p(z)\leq 3$ implies that $p(y_t)\geq 8$, which forces $p$ to be $r$-solvable.

Following similar logic, $p(y_j)=p(y_k)=1$ implies that $p(z)\leq 7$, $p(z)\leq 7$ implies that $p(y_t)\geq 4$, $p(y_t)\geq 4$ implies that $p(z)\leq 1$. Finally, $p(z)\leq 1$ implies that $p(y_t)\geq 8$, forcing $p$ to be $r$-solvable.\\

Notice that $w(p)$ is a multiple of $\frac{1}{3}$ and therefore we are left with only two possibilities.\\

\textbf{Case 5:} Suppose that $w(p)=15\frac{1}{3}$ or $15\frac{2}{3}$.

\textbf{Subcase 5.1}: Suppose that $p(x_i)=p(x_j)=1$ for $x_i,x_j\in X$ and the other vertex in $X$ has $0$ pebbles.

First, suppose that $w(p)=15\frac{1}{3}$. It follows that $w(Y)+w(z)=11\frac{1}{3}$. If $p(y_k)=1$ and $p(z)=10$ for some $y_k\in Y$, we can move $5$ pebbles from $z$ to $y_k$, then move $3$ pebbles to any of $y_k$'s parents with $1$ pebble and it follows that $p$ is $r$-solvable. If $p(Y)=4$ and $p(z)=6$, clearly $p(y_1)= 2$. If $p(y_1)=p(y_2)=2$, move $1$ pebble from $\{y_1,y_2\}$ to each of $\{x_i,x_j\}$, and it follows that $p$ is $r$-solvable. If $p(y_2)<2$, then $p(y_2)=p(y_3)=1$. We can move $1$ pebble from $z$ to each of $\{y_2,y_3\}$ and then $1$ pebble from $\{y_2,y_3\}$ to each of $\{x_i,x_j\}$, which forces $p$ to be $r$-solvable. If $p(Y)=7$ and $p(z)=2$, clearly $p(y_1)\leq 5$, so $p(y_2)\geq 1$. We can move $1$ pebble from $z$ to $y_2$ and then move $1$ pebble from $\{y_1,y_2\}$ to each of $\{x_i,x_j\}$, which implies that $p$ is $r$-solvable.

Second, suppose that $w(p)=15\frac{2}{3}$. If $p(Y)=2$ and $p(z)=9$, remove $4$ pebbles from $z$ to add $1$ pebble on $x_i$ and remove another $4$ pebbles from $z$ to add $1$ pebble on $x_j$. It follows that $p$ is $r$-solvable. If $p(Y)=5$ and $p(z)=5$, by the pigeonhole principle $p(y_1)\geq 2$. Remove $2$ pebbles from $y_1$ and add $1$ pebble to one of $y_1$'s parent from $\{x_i,x_j\}$ and remove $4$ pebbles from $z$ to add $1$ pebble to the other vertex of $\{x_i,x_j\}$. This shows $p$ is $r$-solvable. If $p(Y)=8$ and $p(z)=1$, observe that when $p(Y)\geq 6$ and all $p(y_i)$ is even, $p$ is solvable. When $p(Y)=8$, since at most $2$ vertices of $Y$ have an odd number of pebbles, ignoring $1$ pebble from each vertex of the form $y_i$ with odd pebbles on it implies that $p$ is $r$-solvable.\\

\textbf{Subcase 5.2:} Suppose that $p(x_i)=1$ for some $x_i\in X$, and the other two vertices of $X$ have $0$ pebbles.

Let $n=\left \lfloor \frac{p(z)}{2}\right \rfloor$ and $m=p(Y)$. Let $y_j,y_k,y_t\in Y$ such that $x_i\in pa(y_j)\cap pa(y_k)$ and $x_i\not \in pa(y_t)$. Notice that we can remove $2n$ pebbles from vertex $z$ and distribute $n$ pebbles $Y$.

First, suppose that $n\geq 2$. It follows that $$p(y_j)+p(y_k)\leq 5-n,p(y_j)+p(y_t)\leq 7-n,\textnormal{ and }p(y_k)+p(y_t)\leq 7-n.$$ Then $2m\leq 19-3n$. If $w(p)=15\frac{1}{3}$, then $n=\frac{p(z)}{2}$ and $2m=20-3n$. It follows that $20\leq 19$, forming a contradiction. If $w(p)=15\frac{2}{3}$, then $p(z)= 2n+1$ and $2m=\frac{41-3p(z)}{2}\geq 19-3n$. Thus $2m=19-3n$, which implies that every inequality here takes equality. Thus $$p(y_j)=p(Y)-(p(y_k)+p(y_t))=\frac{19-3n}{2}-7+n=\frac{5-n}{2}.$$ Similarly $p(y_k)=\frac{5-n}{2}$ and $p(y_t)=\frac{9-n}{2}$. But $p(z)$ can only be $7$ or $11$. Readers can easily verify that $p(y_j)=p(y_k)=1,p(y_t)=3,p(z)=7$ and $(y_j)=p(y_k)=0,p(y_t)=2,p(z)=11$ are two $r$-solvable configurations.

Second, suppose that $n<2$. If $w(p)=15\frac{1}{3}$, then $p(Y)=10$ and $p(z)=0$. By Subcase $4.3$, $p(Y)$ is at most $9$, forming a contradiction. If $w(p)=15\frac{2}{3}$, then $p(Y)=8$ and $p(z)=3$. Thus from $z$ we can move a pebble to $Y$. Observe that when $p(Y)=8$ and $p(y_a)$ is even for each $y_a\in Y$, the configuration is $r$-solvable. When $p(Y)=8$, either each $p(y_a)$ is even or exactly $2$ of them are odd. In the latter case, we can move one pebble from $z$ to a vertex $y_a\in Y$ where $p(y_a)$ is odd and ignore a pebble on another vertex $y_b\in Y$ where $p(y_b)$ is odd. This shows $p$ is $r$-solvable.\\

Since all possible configurations have been considered, we conclude that $w$ is valid.
\qed
\end{proof}

\begin{theorem}\label{t2}
Let $V(Q_4)=\{(x_1,x_2,x_3,x_4)|x_1,x_2,x_3,x_4\in \{0,1\}\}$, as shown in Figure~\ref{f5}. Let $w^*$ be the weight function on $Q_4$ defined by $w^*(v)=4$ for each $v\neq (0,0,0,0)$, with $(0,0,0,0)$ being the root. Then, $w^*$ is valid.
\end{theorem}

\begin{figure}[H]
\centering
\bt[scale=0.83]
\centering
\foreach \x in {0,45,90,135,180,225,270,315}
{

	\draw(\x-90:3)--(\x+45:3)--(\x:7.24)--(\x-45:3);
	\draw(\x:7.24)--(\x-45:7.24);
}
\foreach \x in {0,45,135,180,225,270,315}
{
	\draw[fill](\x:7.24)circle(2pt);
	\draw[fill](\x:3)circle(2pt);

}
\draw[fill](90:7.24)circle(2pt);
\draw[fill](90:3)circle(2pt);

\draw (90:7.3)node[anchor=-110] {$(0,0,0,0):r$};
\draw (90:3.2)node[anchor=-110] {$(1,0,0,1)$};
\draw (89.5:-7.38)node[anchor=110] {$(1,1,1,1)$};
\draw (90:-3.2)node[anchor=110] {$(0,1,1,0)$};
\draw (135:7.38)node[anchor=-40] {$(1,0,0,0)$};
\draw (45:7.38)node[anchor=-140] {$(0,0,0,1)$};
\draw (-135:7.38)node[anchor=40] {$(1,1,1,0)$};
\draw (-45:7.38)node[anchor=140] {$(0,1,1,1)$};
\draw (180:7.3)node[anchor=-20] {$(1,1,0,0)$};
\draw (0:7.3)node[anchor=-160] {$(0,0,1,1)$};
\draw (45:3.3)node[anchor=-150] {$(0,0,1,0)$};
\draw (135:3.3)node[anchor=-30] {$(0,1,0,0)$};
\draw (-45:3.3)node[anchor=150] {$(1,0,1,1)$};
\draw (-135:3.3)node[anchor=30] {$(1,1,0,1)$};
\draw (0:3.3)node[anchor=-160] {$(0,1,0,1)$};
\draw (180:3.3)node[anchor=-20] {$(1,0,1,0)$};

\draw(-90:8.4)node{\large $Q_4$};
\et
    \caption{\label{f5}A useful drawing of $Q_4$}

\end{figure}

\begin{proof}
This proof follows from the same technique as Proposition~\ref{p1}. Notice that $w^*$ is the addition of $4$ copies of the weight function $w$ in Lemma~\ref{l5}. In particular, consider vertex sets 
\begin{align*}
V_1=\{(1,x_2,x_3,x_4)|x_2,x_3,x_4\in \{0,1\}\}\cup\{(0,0,0,0)\},\\
V_2=\{(x_1,1,x_3,x_4)|x_1,x_3,x_4\in \{0,1\}\}\cup\{(0,0,0,0)\},\\
V_3=\{(x_1,x_2,1,x_4)|x_1,x_2,x_4\in \{0,1\}\}\cup\{(0,0,0,0)\},\\
V_4=\{(x_1,x_2,x_3,1)|x_1,x_2,x_3\in \{0,1\}\}\cup\{(0,0,0,0)\}.
\end{align*}
Then, the $4$ induced subgraphs of $V_1,V_2,V_3$ and $V_4$, with $(0,0,0,0)$ being root on each of them, are the graphs that these $4$ copies of $w$ lie on. Thus by Lemma \ref{l4}, $w^*$ is valid.
\qed
\end{proof}
\begin{cj}\label{c1}
For any positive integer $n\geq 3$, let $G_n$ be the graph obtained by attaching the root $r$ to an arbitrary vertex of $Q_{n-1}$ through an extra pendent edge. Let $w_n$ be the weight function on $G_n$ where $w_n(r)=0$ and $w_n(v)=\frac{1}{d(v,r)}$ for all $v\neq r$. Then, $w_n$ is valid.
\end{cj}
Let's see an example of a potential application of the weight function defined in Conjecture \ref{c1}. Chung proved in \cite{Chung} that for each integer $n\geq 2$, $\pi(Q_n)=2^n$. The truth of our conjecture implies Chung's result. For any hypercube with dimension at least $3$, let the root $r$ be the vertex where each coordinate is $0$. For each integer $i$ from $1$ to $n$, let $w_{(n,i)}$ be a copy of $w_{n}$ on $H_i$, where $H_i\cong G_n$ is the subgraph of $Q_n$ induced by the vertex set $$V_i=\{(x_1,x_2,\ldots,x_n)|x_i=1\}\cup\{(0,0,\ldots,0)\}.$$ Define $w'_n$ as the weight function on $Q_n$ such that $$w'_n=\sum_{i=1}^{n}w_{(n,i)}.$$ For an arbitrary vertex $v\neq r$ in $Q_{n}$, it has $d(v,r)$ entries being $1$. Thus the number of $H_i$'s containing $v$ is $d(v,r)$. It follows that $$w'_n(v)=\frac{1}{d(v,r)}\cdot d(v,r)=1$$ for each $v\neq r$. By Lemma \ref{l4}, $w'_n$ is valid. By Lemma~\ref{l1}, $\pi(Q_n)\leq 2^n$. Also, the diameter of $Q_n$ is $n$, so $\pi(Q_n)\geq 2^n$ by Lemma~\ref{l2}. Hence $\pi(Q_n)=2^n$.\\

It is worth noting that because $\pi(Q_n)=2^n$ from Chung \cite{Chung}, the validity of the weight function $w'_n$, where $w'_n(v)=1$ for each $v\neq r$ and $w'_n(r)=0$, is a direct result. This is because for any $r$-unsolvable configuration $p$, $\pi(Q_n)=2^n>|p|=w'_n(p)$, where $|p|$ is the number of pebbles in $p$. Since $w'_n(p)$ is an integer, we know $w'_n(1_G)=2^n-1\geq w'_n(p)$ so $w'_n$ is valid. This property is true for any \textbf{Class 0} graph; that is, for any graph whose pebbling number equals its number of vertices.

\section{Weight Functions on Variations of Lollipop Graphs}
\textbf{Lollipop} graphs are formed by attaching the endpoint of a path to some vertex in a cycle. In the context of pebbling, typically the root is placed at the other endpoint of that path. The pebbling numbers and weight functions related to such graphs were first studied in \cite{cranston2015modified}. In this section, we construct a set of valid weight functions on variations of lollipop graphs.

\begin{theorem}\label{t3}
Fix an arbitrary positive integer $n$. The graph in Figure \ref{f6}, denoted by $L_n$, is formed by connecting a path of length $n$ to a collection of $2^{n+1}$ paths of length $2$ that all share the same starting and ending vertex. Let $w$ be the weight function on $L_n$, as shown in Figure~\ref{f6}, where $w$ is defined by $w(u_i)=1$ for each $i\in \{0,1,\ldots,2^{n+1}\}$ and $w(v_j)=2^j$ for each $j\in \{1,2,\ldots,n\}$, and $w(r)=0$. Then, $w$ is valid.
\end{theorem}

\begin{figure}[H]
\centering
\begin{tikzpicture}[scale=1]
\begin{scope}[rotate=-90,yshift=-3.6cm]
\draw[fill] (-3,0)node[anchor=-90]{$r$}circle(2pt)--(-2,0)node[anchor=180]{$v_n$}circle(2pt)--(-1,0)node[anchor=180]{$v_{n-1}$}circle(2pt);
\draw (1,0)--(2.5,2.3)node[anchor=180]{$u_{2^{n+1}}$}--(4,0)--(2.5,0.6)node[anchor=180]{$u_{2^{n+1}-1}$}--(1,0)--(2.5,-0.6)node[anchor=0]{$u_2$}--(4,0)--(2.5,-2.3)node[anchor=0]{$u_1$}--(1,0);
\draw[fill](0,0)node[anchor=180]{$v_2$}circle(2pt)--(1,0)node[anchor=180]{$v_1$}circle(2pt);
\node[dot] at (2.53,0) {$\cdots$};
\node[dot2,font=\large] at (-0.5,0){$\cdots$};
\foreach \x in {1,4}
{
	\draw[fill] (\x,0)circle(2pt);
}
\foreach \x in {2.3,0.6,-0.6,-2.3}
{
	\draw[fill](2.5,\x)circle(2pt);
	
}
\draw[fill] (4,0)node[anchor=90] {$u_0$}circle(2pt);

\draw(4.8,0)node{\large $L_n$};
\end{scope}

\begin{scope}[rotate=-90,yshift=3.6cm]
\draw[fill] (-3,0)node[anchor=-90]{$r$}circle(2pt)--(-2,0)node[anchor=180]{$2^n$}circle(2pt)--(-1,0)node[anchor=180]{$2^{n-1}$}circle(2pt);
\draw (1,0)--(2.5,2.3)node[anchor=180]{$1$}--(4,0)--(2.5,0.6)node[anchor=180]{$1$}--(1,0)--(2.5,-0.6)node[anchor=0]{$1$}--(4,0)--(2.5,-2.3)node[anchor=0]{$1$}--(1,0);
\draw[fill](0,0)node[anchor=180]{$2^2$}circle(2pt)--(1,0)node[anchor=196]{$2$}circle(2pt);
\node[dot] at (2.53,0) {$\cdots$};
\node[dot2,font=\large] at (-0.5,0){$\cdots$};
\foreach \x in {1,4}
{
	\draw[fill] (\x,0)circle(2pt);
}
\foreach \x in {2.3,0.6,-0.6,-2.3}
{
	\draw[fill](2.5,\x)circle(2pt);
	
}

\draw[fill] (4,0)node[anchor=90] {$1$}circle(2pt);

\draw(4.8,0)node{\large $w$};
\end{scope}
\end{tikzpicture}
    \caption{\label{f6}A valid weight function on $L_n$, denoted by $w$}

\end{figure}

\begin{proof}
Let $P$ be the path $(v_1,v_2,\ldots,v_n,r)$ and $U$ be the vertex set $\{u_0,u_1,u_2,\ldots,u_{2^{n+1}}\}$. Let $p$ be an arbitrary $r$-unsolvable configuration on $L_n$. Note that for each $v_i\in P$ with $i\neq n$, we have $w(v_{i+1})=2w(v_i)$. Let $p$ be an arbitrary configuration on $L_n$ with $$w(p)\geq w(1_{L_n})+1=2^{n+2}.$$ To prove $w$ is valid, it suffices to show $p$ is $r$-solvable. For purposes of contradiction suppose that $p$ is not $r$-solvable. By definitions in Section $1.1$, we have the following:
\begin{align*}
w(p)&=\sum_{v\in V(L_n)}p(v)\cdot w(v)\\
w_P(p)&=\sum_{v\in V(P)}p(v)\cdot w(v)\\
p(U)&=\sum_{v\in U}p(v).
\end{align*}
Since $w(v_1)=2$, applying Lemma \ref{l2} on $P$ shows that if $w_P(p)\geq 2^{n+1}$ then $p$ is $r$-solvable. Suppose $w_P(p)<2^{n+1}$. For simplicity of notation, let $w_P(p)=x$ and $p(u_0)=2b+y$ where $b\geq 0$ and $y$ equals $0$ or $1$. Let the number of vertices in $U-\{u_0\}$ with an odd number of pebbles be $a$. Clearly $x$ is even and $a\leq 2^{n+1}$. Since $w(u)=1$ for each $u\in U$, we know $$p(U)=w(U)={w(p)-w_P(p)\over 1}\geq 2^{n+2}-x.$$ Lemma \ref{l2} shows that it suffices to prove we can move at least $\frac{2^{n+1}-x}{2}$ pebbles from $U$ to $v_1$.\\

\textbf{Case 1: }Suppose that $a\leq b$.

We will make the following pebbling moves in sequence:
\begin{itemize}
\item Move $1$ pebble from $u_0$ to each vertex in $U-\{u_0\}$ that has an odd number of pebbles. Now each vertex in $U-\{u_0\}$ has an even number of pebbles.
\item For each vertex in $U-\{u_0\}$, move pebbles from it to $v_1$ until there are no pebbles on it.
\item If $u_0$ still has at least $4$ pebbles, move two pebbles from $u_0$ to $u_1$ and then move $1$ pebble from $u_1$ to $v_1$. Repeat this step until $u_0$ has fewer than $4$ pebbles.
\end{itemize}
In the pebbling moves described above, there are three ways for us to add a pebble onto $v_1$: use $2$ pebbles from $u_0$ and $1$ pebble from some vertex in $U-\{u_0\}$, use $2$ pebbles from some vertex in $U-\{u_0\}$, or use $4$ pebbles from $u_0$. Hence we conclude that adding $1$ pebble onto $v_1$ uses at most $4$ pebbles.

Observe that after these pebbling moves, each vertex in $U-\{u_0\}$ has no pebbles and $u_0$ has at most $3$ pebbles. Thus at least $p(U)-3$ pebbles are consumed in order to add pebbles onto $v_1$. Hence, the number of pebbles we added to $v_1$ through these pebbling moves is at least$$\left \lceil {p(U)-3\over 4}\right \rceil\geq \left\lceil {2^{n+2}-x-3\over 4}\right\rceil\geq {2^{n+1}-x\over 2}$$where the last inequality holds because $x$ is even.\\

\textbf{Case 2: }Suppose that $a>b$.

We will make the following pebbling moves in sequence:
\begin{itemize}
\item Pick $b$ vertices in $U-\{u_0\}$ that have an odd number of pebbles. Move $1$ pebble from $u_0$ to each of them.
\item For each vertex in $U-\{u_0\}$, move pebbles from it to $v_1$ until it has at most $1$ pebble.
\end{itemize}
Similar to Case 1, the number of pebbles we add to $v_1$ through these pebbling moves is exactly $${p(U)-a-y\over 2}\geq \left\lceil {2^{n+2}-a-x-y\over 2}\right\rceil={2^{n+1}-x\over 2}+\left\lceil{(2^{n+1}-a-y)\over 2}\right\rceil\geq {2^{n+1}-x\over 2},$$where the last inequality holds because $a+y\leq 2^{n+1}+1$.\\

Hence in all cases we can move at least ${2^{n+1}-x\over 2}$ pebbles from $U$ to $v_1$ and we conclude that $w$ is valid.

\qed
\end{proof}

In fact, this weight function can be further generalized to contain $m$ vertices of the form $u_i$, where $m=|U|\geq 2^{n+1}+1$. Here is the brief idea for proving its validity. Denote this graph as $G$ and let $p$ be an arbitrary configuration on $G$ such that $w(p)>w(1_G)$. Then, we first move pebbles from each vertex in $U-\{u_0\}$ to $v_1$ until at most $1$ pebble is left on it. This gives us a new configuration $p'$. Clearly $w(p)=w(p')$. Then, ignore $m-2^{n+1}$ arbitrary vertices in $U-\{u_0\}$ and let the subgraph induced by the rest of the vertices be denoted by $G'$. Notice that $w_{G'}(p')\geq w(p)-1\cdot (m-2^{n+1})>w(1_{G'})$. By Theorem \ref{t3}, configuration $p'$ on $G'$ is $r$-solvable, which implies that configuration $p$ on $G$ is also $r$-solvable.

\section{Open questions}

%We conclude with some ideas for further research.

\begin{enumerate}
    \item Can more efficient proof techniques be developed to prove the validity of weight functions?
    \item Is it possible to modify weight functions to characterize other variations of pebbling problems? For example, how could techniques like weight functions be used to characterize configurations that could place at least $2$ pebbles on the root vertex $r$? The authors believe that such modifications could help prove Conjecture \ref{c1}.
    \item In this paper we reviewed weight functions for odd cycles. Generalized lollipop graphs with odd cycles might involve a path connected to an odd cycle or a path connected to a set of paths whose lengths might differ by one that share the same starting and ending vertex. Can we find valid weight functions for generalized lollipop graphs with odd cycles analogous to Lemma 3 in \cite{cranston2015modified}? Specifically, could the paths at the end of the lollipop have length larger than $2$? What if there were multiple collections of short paths along the long path instead of just one at the end? Note that weight functions on such graphs are used to prove Theorem 7 in \cite{cranston2015modified}.
    \item Recall that a \textit{Lemke graph} is a graph that does not satisfy the 2-pebbling property.  One possible path to a counterexample of Graham's conjecture is to investigate the Cartesian product of two copies of a Lemke graph.  Specifically, there have been work \cite{flocco,hurlbert2017,kenter} on bounding the pebbling number of the original 8-vertex graph of Lemke (see \cite{sne}).  Let this graph be denoted $L$. Can we use these new weight functions to improve the bound of the pebbling number of $L\mathbin{\square} L$  \cite{cranston2015modified}?
\end{enumerate}

Note that our results have more implications than simply constructing other valid weight functions or computing the pebbling numbers of a few graphs. We believe that they can be coupled with computational tools (extending work of Flocco et al. \cite{flocco}) to improve computational techniques for computing pebbling numbers of any graph that has hypercubes or lollipops as induced subgraphs.

\section*{Data Availibility}

The manuscript has no associated data.

\section*{Conflict of Interest}

The authors declare that they have no conflict of interest.

\section*{Acknowledgement}
The authors would like to thank an anonymous referee for their careful reading of the manuscript and helpful comments.

\textbf{Funding} This work was partially supported by the Davidson Research Initiative from Davidson College.

\bibliographystyle{abbrv}      % mathematics and physical sciences
\bibliography{lollipop}

\end{document}